\documentclass[11pt]{article}
\usepackage{amssymb}
\usepackage{geometry,amsmath,amsfonts,amssymb,amsthm,color,mathrsfs,indentfirst}
\usepackage[]{hyperref}
\numberwithin{equation}{section}
\newtheorem{theorem}{Theorem}[section]
\newtheorem{lemma}{Lemma}[section]

\newtheorem{remark}{Remark}[section]

\allowdisplaybreaks

\begin{document}

\title{\textbf{Logistic damping effect in chemotaxis models with density-suppressed motility}}
\author{Wenbin Lyu \thanks{%
School of Mathematical Sciences, Shanxi University, Taiyuan 030006, P.R. China; lvwenbin@sxu.edu.cn. }  \and Zhi-An Wang
\thanks{%
Department of Applied Mathematics, The Hong Kong Polytechnic University, Hong Kong; mawza@polyu.edu.hk. }}
\date{}
\maketitle

{\noindent\bf Abstract.}
This paper is concerned with a parabolic-elliptic chemotaxis model with density-suppressed motility and general logistic source in an $n$-dimensional smooth bounded domain with Neumann boundary conditions. Under the minimal conditions for the density-suppressed motility function, we explore how strong the logistic damping can warrant the global boundedness of solutions, and further establish the asymptotic behavior of solutions on top of the conditions.
\vspace{0.2cm}

\vspace{0.2cm}



\section{Introduction and main results}

To explain the aggregation phase of Dictyostelium discoideum cells in response to the secreted chemical signal cyclic adenosine monophosphate, Keller and Segel \cite{K-S1971JTB} proposed the following well-known system
\begin{equation}\label{28}
\begin{cases}
u_t=\nabla\cdot(\gamma (v)\nabla u-u\phi(v)\nabla v),&x\in\Omega,\,t>0,\\
\tau v_t=\Delta v-v+u,&x\in\Omega,\,t>0,
\end{cases}
\end{equation}
in 1971, where $\Omega\subset\mathbb{R}^n\,(n\geq1)$ is a bounded domain with smooth boundary, $u$ denotes the cell density and $v$ is the concentration of the chemical signal emitted by cells. The parameter $\tau\in\{0,1\}$ represents the
relaxation time, that is, the rate of the time scale of $v$ relative to $u$. $\gamma(v)$ and $\phi(v)$ are motility functions representing diffusive and chemotactic coefficients, respectively, and both of them depend on the chemical signal concentration linked through the following relation
$$\phi(v)=(\alpha-1)\gamma'(v)$$
and $\alpha$ denotes the ratio of effective body length (i.e. distance between receptors) to step size. Of particular interest is the case $\alpha=0$, namely the distance between receptors is zero, where the chemotaxis occurs because of an undirected effect on activity due to the presence of a chemical sensed by a single receptor (cf. \cite [p.228]{K-S1971JTB}), the above system is reduced to
\begin{equation}\label{5}
\begin{cases}
u_t=\Delta(\gamma (v)u),&x\in\Omega,\,t>0,\\
\tau v_t=\Delta v-v+u,&x\in\Omega,\,t>0.
\end{cases}
\end{equation}

When $\gamma(v)$ and $\phi(v)$ are constant, the model \eqref{28} is called the minimal chemotaxis model (cf. \cite{Nanjundiah}), which has been extensively studied in the literature. For a broad overview on various types of chemotactic processes and relevant mathematical results, we refer the readers to the survey papers \cite{H-P2009JMB,H.D.2003,T.S.2005,T.S.2018} and references therein.
In contrast to abundant results obtained for the minimal chemotaxis system, the progresses made to the original Keller-Segel model \eqref{28} with non-constants $\gamma(v)$ and $\phi(v)$ are very limited.

To describe the stripe pattern formation of bacterial movement observed in the experiment of \cite{L.2011S}, a so-called density-suppressed motility model as below  was proposed in \cite{L.2011S} (see the supplemental material) and formally analyzed in \cite{F.2012PRL}
\begin{equation}\label{OR-2}
\begin{cases}
u_t=\Delta(\gamma(v)u)+\mu u(1-u),&x\in\Omega, t>0,\\
\tau v_t=\Delta v+u-v,& x\in \Omega, t>0,
\end{cases}
\end{equation}
where the parameter $\mu>0$ denotes the intrinsic cell growth rate, and $\gamma'(v)<0$ accounting for the repressive effect of chemical signal concentration on the cell motility (cf. \cite{L.2011S}). Of interest is that \eqref{OR-2} with $\mu=0$ coincides with the simplified Keller-Segel model \eqref{5}. Indeed the density-suppressed motility mechanism has been previously used to model other biological processes, such as  the predator-prey system describing  the inhomogeneous co-existence distributions of ladybugs (predators) and aphids (prey) populations in the field (cf. \cite{J-W2021EJAM,K-O1987TAN,W-X2021JMB}).
Under the condition $\gamma'(v)<0$ (i.e. density-suppressed motility),   the main challenge of the analysis lies in the possible diffusion degeneracy since $\gamma(v)$ could have no positive lower bound. Therefore many conventional methods for reaction-diffusion equations or chemotaxis models are inapplicable. The dynamics of the above-mentioned models with density-suppressed motility have not been well understood until recently.  Below we shall give a brief review of existing results on \eqref{28}, \eqref{5} and \eqref{OR-2} and then raise our question to explore.  In what follows, we shall always assume the homogenous Neumann boundary conditions unless otherwise stated.

{\bf Case of $\mu>0$}. It was first shown in \cite{J-K-W2018SIAM} when the motility function $\gamma(v)$ satisfies
$\gamma(v)\in C^3([0,+\infty))$, $\gamma(v)>0,\ \gamma'(v)<0\ \mathrm{on}\ [0,\infty)$, $\lim\limits_{v \to +\infty}\gamma(v)=0$ and $\lim\limits_{v \to +\infty}\frac{\gamma'(v)}{\gamma(v)}$ exists, system \eqref{OR-2} has a unique global classical solution in two dimensions ($n=2$) which globally asymptotically converges to the equilibrium $(1,1)$  if
$\mu>\frac{K_0}{16}$ with $K_0=\max\limits_{0\leq v \leq +\infty}\frac{|\gamma'(v)|^2}{\gamma(v)}$. Later, similar results were extended to higher dimensions ($n\geq 3$) for large $\mu>0$ in \cite{L-X2019JMAA} and \cite{W-W2019JMP}. The condition that $\lim\limits_{v \to +\infty}\frac{\gamma'(v)}{\gamma(v)}$ exists imposed in \cite{J-K-W2018SIAM} was recently removed in \cite{F-J2020JDE,J-W2021DCDSB} for the parabolic-elliptic case model (i.e. $\tau=0$). On the other hand, for small $\mu>0$, the existence/{\color{black}nonexistence} of nonconstant steady states  of \eqref{OR-2} was rigorously established in \cite{M-P-W2020PD,X-W2021IMAJAM} in appropriate parameter regimes and the periodic pulsating wave is analytically obtained by the multi-scale analysis.

{\bf Case of $\mu=0$}. The analysis for the case $\mu=0$ was much more delicate due to the loss of logistic damping.  It was first shown in \cite{T-W2017MMMAS} that globally bounded solutions exist in two dimensions if the motility function $\gamma(v)$ has both positive lower and upper bounds. However, $\gamma(v)$ may not have priori positive lower/upper bound, for example, $\gamma(v)=\frac{c_0}{v^k}$, for which it was proved in \cite{Y-K2017AAM} that global bounded solutions exist in all dimensions for any $k>0$ provided that $c_0>0$ is small. The smallness of $c_0$ is later removed in \cite{A-Y2019N} for the parabolic-elliptic case model (i.e. $\tau=0$ in \eqref{OR-2}) for $0<k<\frac{2}{(n-2)_+}$. The global existence of weak solutions of \eqref{OR-2} with large initial data was established in \cite{D-K-T-Y2019NARWA} for $\gamma(v)=\frac{1}{c+v^k}$ with $c\geq 0,\ k>0$ and $n=1$ or $k\in(0,2)$ and $n=2$ or $k\in(0,\frac{3}{4})$ and $n=3$. When $\gamma(v)=e^{-\chi v}$, a critical mass phenomenon was identified in \cite{J-W2020PAMS}:  if $n=2$, there is a critical number $m=4\pi/\chi>0$ such that the solution of \eqref{OR-2} with $\tau=1$ may blow up if the initial cell mass $\|u_0\|_{L^1(\Omega)}>m$, while global bounded solutions exist if $\|u_0\|_{L^1(\Omega)}<m$. This result was further refined in \cite{F-J2020JDE,F-J2021CVPDE} showing that the blowup occurs at the infinity time. Meanwhile, it was proved in \cite{F-J2020pre,Z.A.W.2021MMAS,W-Z2021AAM} that global bounded classical solutions exist in all dimensions with some additional conditions on $\gamma(v)$. Recently global weak solutions of \eqref{OR-2} in all dimensions and the blow-up of solutions of \eqref{OR-2} in two dimensions were investigated in \cite{B-L-T2021JLMS}.

We mention that an addition to the two-component density suppressed motility model \eqref{OR-2},  a three-component density-suppressed motility model was also proposed in  \cite{L.2011S} and has been studied recently in \cite{J-L-Z2021pre,J-S-W2020JDE,Lyu-W, X-W2021CVPDE}. Except the studies on the bounded domain with zero Neumann boundary conditions, there are some results obtained in the whole space $\mathbb{R}$. When $\gamma(v)$ is a piecewise constant function, the dynamics of discontinuity interface of \eqref{OR-2} was studied in \cite{S-I-K2019EJAM} and  discontinuous traveling wave solutions were constructed in \cite{L-N2019DCDSB}.  Recently the existence of traveling wavefronts of \eqref{OR-2}  was established in \cite{L-W2021JDE} where the periodic traveling waves were also numerically illustrated.

The present work is motivated in the following way. When $\mu=0$, the solution of \eqref{OR-2} (i.e. the model \eqref{5})  may blow up in two dimensions with the motility function $\gamma(v)$ decaying exponentially (cf. \cite{J-W2020PAMS,F-J2020JDE,F-J2021CVPDE}), while the blowup is immediately prevented by the quadratic logistic damping $\mu u(1-u)$ with $\mu>0$ (cf. \cite{J-K-W2018SIAM, F-J2020JDE,J-W2021DCDSB}). Hence we ask how strong the logistic damping is actually adequate to preclude the blowup of solutions for the density-suppressed motility function satisfying only the basic (minimal) assumptions in multi-dimensions. This question amounts to consider the system \eqref{OR-2} by replacing  $\mu u(1-u)$ with $\mu(u-u^\sigma)$ and explore how large the damping exponent $\sigma>0$ can ensure the global boundedness of solutions.
Therefore we are motivated to consider the following system
\begin{equation}\label{1}
\begin{cases}
u_{t}=\Delta(\gamma(v)u)+au-bu^\sigma, & x \in \Omega, \quad t>0, \\
-\Delta v+v=u, & x \in \Omega, \quad t>0,\\
\frac{\partial u}{\partial \nu}=\frac{\partial v}{\partial \nu}=0, & x \in \partial \Omega, \quad t>0, \\
u(x, 0)=u_{0}(x), & x \in \Omega,
\end{cases}
\end{equation}
where constants $a>0,\ b>0,\ \sigma>1$ and $\nu$ is the unit outer normal vector of $\partial\Omega$.

We imposed the minimal structural assumptions on the density-suppressed motility function
\begin{equation}\label{11}
\gamma\in C^3([0,+\infty)),\ \gamma(v)>0,\ \gamma'(v)<0 \quad \text{on}\ [0,+\infty),
\end{equation}
and assume the initial value $u_0$ satisfies
\begin{equation}\label{29}
u_0\in C^0(\overline{\Omega}),\ u_0\geq0\ \text{and}\ u_0\not\equiv0.
\end{equation}
Then we shall investigate under what conditions on the damping exponent $\sigma$ associated with the dimension $n$, the global boundedness of solutions to \eqref{1} is ensured,  and further find the asymptotic behavior of solutions as time tends to infinity.

For the convenience of presentation, we let
\begin{equation}\label{63}
Q:=Q(b,\sigma)=\left\|(-\Delta+I)^{-1}u_0\right\|_{L^\infty(\Omega)}+\frac{\sigma-1}{\gamma(0)b^{\frac{1}{\sigma-1}}}\left(\frac{a+2\gamma(0)}{\sigma}\right)^{\frac{\sigma}{\sigma-1}},
\end{equation}
and $G:=G(p)>0$ be a constant satisfying the following Gagliardo-Nirenberg inequality
  $$\|\nabla v\|_{L^{2(p+1)}(\Omega)} \leq G(p)\|v\|_{W^{2, p+1}(\Omega)}^{\frac12}\|v\|_{L^{\infty}(\Omega)}^{\frac12}.$$
By the Agmon-Douglis-Nirenberg regularity theorem (cf. \cite{Agmon1, Agmon2}) applied to the elliptic equation $-\Delta v+v=u$, we find a constant $R:=R(p)>0$ such that
  $$\|v\|_{W^{2,p}(\Omega)}\leq R(p)\|u\|_{L^p(\Omega)}.$$
We write
  \begin{equation}\label{65}
  \kappa:=\kappa(p,b,\sigma)=\left( G^{2(p+1)}(p)R^{p+1}(p)Q^{p+1}(b,\sigma)+1\right)\sup_{0\leq v\leq Q(b,\sigma)}\frac{|\gamma'(v)|^2}{\gamma(v)}
  \end{equation}
  and
  \[b_1:=\frac{\left[\frac n2\right]\kappa\left(\left[\frac n2\right]+1,1,2\right)}{2}+1.\]
Then our main results are stated as follows.

\begin{theorem}\label{th4}
Let $\Omega\subset \mathbb{R}^n$ be a bounded domain with smooth boundary. Assume $\gamma(v)$ satisfies the hypothesis \eqref{11} and the initial value satisfies the condition \eqref{29}. Then If one of the following conditions holds
\begin{enumerate}
  \item[(1)] $n\leq2$, $\sigma>1$;
  \item[(2)] $n\geq3$, $\sigma>2$;
  \item[(3)] $n\geq3$, $\sigma=2$ and $b>b_1$,
\end{enumerate}
then there exists a couple
$(u,v)$ of non-negative functions
$$u\in C^0(\overline{\Omega}\times[0,+\infty))\cap C^{2,1}(\overline{\Omega}\times(0,+\infty)),\quad v\in C^{2,1}(\overline{\Omega}\times(0,+\infty)),$$
which solves \eqref{1} in the classical sense. Moreover, the solution of \eqref{1} is uniformly bounded in $\Omega\times(0,+\infty)$; namely, there exists a constant $C>0$ independent of $t$ such that
\begin{equation*}
\|u(\cdot,t)\|_{L^\infty(\Omega)}+\|v(\cdot,t)\|_{W^{1,\infty}(\Omega)}\leqslant C\quad\text{for all}\ t>0.
\end{equation*}
\end{theorem}

\begin{remark}\em{
The results of Theorem \ref{th4} indicate that if $\sigma>2$ (strong damping), then the minimal assumption \eqref{11} on the motility function is sufficient to warrant the global boundedness of solution to \eqref{1} in all dimensions. If $1<\sigma<2$ (weak damping), then the  minimal assumption \eqref{11} can guarantee the global boundedness of solutions for $n\leq2$. If $n=2$, $\gamma(v)=e^{-\lambda v}$ and $a=b=0$, a critical mass $m_0=\frac{4\pi}{\lambda}$ exists and solution may blow up at infinite time if the initial cell mass is grater than $m_0$ (cf. \cite{F-J2020JDE,F-J2021CVPDE,J-W2020PAMS}).  Therefore Theorem \ref{th4} entails that the mere logistic damping ($\sigma>1$) can  preclude the blowup in two dimensions, in contrast to the results in \cite{J-K-W2018SIAM, J-W2021DCDSB,F-J2020JDE} with $\sigma=2$. In \cite{L-W2021EECT,L-W2021PRSE}, the global boundedness of solutions to the parabolic-parabolic version of \eqref{1} in all dimensions was obtained for $\sigma>\max\{2,\frac{n+2}{2}\}$. Clearly these results are improved by our results in Theorem \ref{th4} in the parabolic-elliptic case model.
}
\end{remark}


Next, we investigate the large time behavior of solutions to \eqref{1}. For convenience, let
\begin{equation}\label{33}
u_*:=\left(\frac{a}{b}\right)^{\frac{1}{\sigma-1}},\quad\overline{v}:=\frac{1}{|\Omega|}\int_\Omega v,
\end{equation}
\begin{equation}\label{66}
K_1:=\begin{cases}
\frac{(2|\Omega|)^{\frac{\sigma+1}{3-\sigma}}2^{2-\sigma}}{\sigma-1}+\frac{2^{\frac{4}{3-\sigma}}(2|\Omega|)^{2-\sigma}}{1-2^{-(\sigma-1)}},&\text{if}\ \ \sigma<2,\\
1,&\text{if}\ \ \sigma\geq 2,
\end{cases}
\end{equation}
and
\begin{equation}\label{64}
K_2:=\begin{cases}
\xi^2,&\text{if}\ \ \sigma<2,\\
1,&\text{if}\ \ \sigma\geq2,
\end{cases}
\end{equation}
where
$\xi>0$ is a Poincar\'{e} constant satisfying the following inequality
$$\|v-\overline{v}\|_{L^{d}(\Omega)}\leq\xi\|\nabla v\|_{L^2(\Omega)},\quad d\in\left[1,\frac{2n}{(n-2)_+}\right).$$
Then we write
\begin{equation}\label{62}
b_2:=\left(\frac{K_1K_2}{4}\right)^{\frac{\sigma-1}{2}}a^{-\frac{\sigma-3}{2}}\left(\sup_{v\geq0}\frac{|\gamma'(v)|^2}{\gamma(v)}\right)^{\frac{\sigma-1}{2}}.
\end{equation}

\begin{theorem}\label{th3}
Let the assumptions in Theorem \ref{th4} hold.  If one of the following conditions holds
\begin{enumerate}
  \item[(1)] $n\leq2$, $\sigma>1$ and $b>b_2$;
  \item[(2)] $n\geq3$, $\sigma>2$ and $b>b_2$;
  \item[(3)] $n\geq3$, $\sigma=2$ and $b>\max\{b_1,b_2\}$;
  \item[(4)] $n\geq3$, $2-\frac{2}{n}<\sigma<2$, $b>b_2$ and $\gamma(v)$ satisfies $\inf\limits_{v\geq0}\frac{\gamma(v)\gamma''(v)}{|\gamma'(v)|^2}>\frac{n}{2}$,
\end{enumerate}
then the solution of \eqref{1} obtained in Theorem \ref{th4} satisfies
\begin{equation}\label{1.5}
\left\|u(\cdot,t)-u_*\right\|_{L^\infty}\rightarrow0\quad\text{and}\quad\left\|v(\cdot,t)-u_*\right\|_{L^\infty}\rightarrow0\quad\text{as}\quad t\rightarrow+\infty.
\end{equation}
\end{theorem}

The key of proving Theorem \ref{th4} is to derive that $v$ has an upper bound to rule out the possible diffusion degeneracy (see section \ref{sec1}). Inspired by an idea from the work \cite{F-J2020JDE}, we use the maximum principle for the inverse operator $(I-\Delta)^{-1}$ to construct a simple differential inequality on $v$ which yields an upper bound of $v$. In deriving the {\it a priori} estimates to obtain the global solutions, the case $\sigma>2$ is easier to handle due to the strong damping effect, while in the case  $1<\sigma\leq2$, we shall further exploit the structure of the equation governing $v$. To study the large time behaviour of solutions, the Lyapunov functional method and Poincar\'{e}'s inequality are essentially used.

The rest of this paper is organized as follows. In Section \ref{sec2}, we first state the local existence of solutions to \eqref{1} with an extensibility condition, and derive a criterion  of global boundedness of solutions. Then we deduce some {\it a priori} estimates to meet the criterion and obtain the global bounded solutions of \eqref{1} in section \ref{sec3}. Finally, we show the large time behaviour of solutions to \eqref{1} in section \ref{sec4}.

\section{Preliminaries}\label{sec2}

This section is devoted to introducing some preliminary results for proving the global boundedness of solutions to \eqref{1}, including the existence of local solutions and extensibility criterion for  \eqref{1} as well as some basic estimates. In the following, we shall use $C_i(i=1,2,\cdots)$ to denote a generic positive constant which may vary in the context. For simplicity, we shall abbreviate $\int_\Omega fdx$ as $\int_\Omega f$ and $\int_0^t\int_\Omega fdxds$ as $\int_0^t\int_\Omega f$ without confusion.

\subsection{Local existence and extensibility criterion}

First, we present the existence of local solutions and extensibility criterion for \eqref{1} which can be proved via the fixed point theorem by slightly modifying the proof of \cite[Lemma 2.1]{J-K-W2018SIAM}.

\begin{lemma}[Local existence]\label{lm5}
Let $\Omega\subset \mathbb{R}^n$ be a bounded domain with smooth boundary, $a>0$, $b>0$, $\sigma>1$ and assume $\gamma(v)$ satisfies the hypothesis \eqref{11}. If the initial data satisfy the condition \eqref{29}, then there exists $T_{max}\in(0,+\infty]$ such that \eqref{1} admits a unique non-negative classical solution $(u,v)$ satisfying
$$u\in C^0(\overline{\Omega}\times[0,T_{max}))\cap C^{2,1}(\overline{\Omega}\times(0,T_{max})),$$
$$v\in C^{2,1}(\overline{\Omega}\times(0,T_{max})).$$
Moreover, if $T_{max}<+\infty$, then
\begin{equation}\label{7}
\limsup\limits_{t\nearrow T_{max}}\|u(\cdot,t)\|_{L^\infty(\Omega)}=+\infty.
\end{equation}
\end{lemma}
Below is an uniform Gr\"{o}nwall inequality \cite{R.T.1988} which will help us to derive the uniform in-time estimates of solutions.

\begin{lemma}\label{lm9}
Let $T_{max}>0$, $\tau\in(0,T_{max})$. Suppose that $c_1,c_2,y$ are three positive locally integrable functions on $(0,T_{max})$ such that $y'$ is locally integrable on $(0, T_{max})$ and satisfies
$$y'(t) \leq c_1(t)y(t)+c_2(t)\ \ \text{for all}\ \ t\in(0,T_{max}).$$
If
$$\int_{t}^{t+\tau} c_1\leq C_{1}, \ \ \int_{t}^{t+\tau} c_2 \leq C_{2}, \ \  \int_{t}^{t+\tau} y \leq C_{3}\ \ \text{for all}\ \ t\in[0, T_{max}-\tau)$$
where $C_{i}(i=1,2,3)$ are positive constants, then
$$y(t) \leq\left(\frac{C_{3}}{\tau}+C_{2}\right) e^{C_{1}}\ \ \text{for all}\ \ t\in[\tau, T_{max}).$$
\end{lemma}

\subsection{The uniform in-time \texorpdfstring{$L^1$}{L1} boundedness of \texorpdfstring{$u$}{u}}

The uniform in-time $L^1$ boundedness of $u$ below is a basic property of solutions and will pays a role in our analysis.

\begin{lemma}\label{lm3}
Let $(u,v)$ be a solution of \eqref{1}. If $a>0,\,b>0$ and $\sigma>1$ are constants, then there exists a constant $C>0$ such that
\begin{equation}\label{39}
\int_\Omega u\leq C\quad\text{for all}\ t\in(0,T_{max}).
\end{equation}
\end{lemma}

\begin{proof}
Integrating the first equation of \eqref{1} over $\Omega$ with the boundary condition, one finds
\begin{equation}\label{2}
\frac{d}{dt}\int_\Omega u\leq a\int_\Omega u-b\int_\Omega u^\sigma\quad\text{for all}\ t\in(0,T_{max}).
\end{equation}
Due to $\sigma>1$ and H\"{o}lder's inequality, we conclude the fact
$$\int_\Omega u\leq|\Omega|^{\frac{\sigma-1}{\sigma}}\left(\int_\Omega u^\sigma\right)^{\frac{1}{\sigma}},$$
which, along with \eqref{2}, implies that
\begin{equation*}
\frac{d}{dt}\int_\Omega u\leq a\int_\Omega u-\frac{b}{|\Omega|^{\sigma-1}}\left(\int_\Omega u\right)^\sigma.
\end{equation*}
Solving this ordinary differential inequality and noticing the positivity of $u$, we have
\begin{equation*}
\int_\Omega u\leq\max\left\{\int_\Omega u_0,\left(\frac{a}{b}\right)^{\frac{1}{\sigma-1}}|\Omega|\right\}.
\end{equation*}
This finishes the proof of \eqref{39}.
\end{proof}

\subsection{The upper bound of \texorpdfstring{$v$}{v}}\label{sec1}

To rule out the possible degeneracy of $\gamma(v)$, we shall derive the upper bound of $v$ in this section. Our method is essentially inspired from the paper \cite{F-J2020JDE}. The following lemma asserts that the
upper bound of $v$ is independent of $T_{max}$.

\begin{lemma}\label{lm1}
Let $(u,v)$ be a solution of \eqref{1}. If $a>0,\,b>0$ and $\sigma>1$ are constants, then there is a constant $C>0$ such that $v$ satisfies
$$v_t+\gamma(v)u+\gamma(0)v\leq C\quad\text{for any}\ \ (x,t)\in\Omega\times(0,T_{max}).$$
Moreover, we have
\begin{equation*}
v\leq Q\quad\text{for any}\ \ (x,t)\in\Omega\times(0,T_{max}),
\end{equation*}
where the constant $Q>0$ is defined in \eqref{63}.
\end{lemma}

\begin{proof}
Step 1: We claim $v$ satisfies the following equation
\begin{equation}\label{6}
v_t+\gamma(v)u=(I-\Delta)^{-1}\{\gamma(v)u+au-bu^\sigma\}
\end{equation}
for all $(x,t)\in\Omega\times(0,T_{max})$. Indeed, the first equation of \eqref{1} can be rewritten as
$$u_t=-(I-\Delta)\gamma(v)u+\gamma(v)u+au-bu^\sigma.$$
Taking the operator $(I-\Delta)^{-1}$ on both sides of the above equation and using the second equation of \eqref{1}, we get \eqref{6} directly.

Step 2: By the non-increasing property of $\gamma(v)$ and Young's inequality, we can find a constant $$C_1:=\frac{\sigma-1}{b^{\frac{1}{\sigma-1}}}\left(\frac{a+2\gamma(0)}{\sigma}\right)^{\frac{\sigma}{\sigma-1}}>0$$ such that
\begin{align*}
\gamma(v)u+au-bu^\sigma\leq&\gamma(0)u+au-bu^\sigma\\
=&-\gamma(0)u+\left(a+2\gamma(0)\right)u-bu^\sigma\\
\leq&-\gamma(0)u+C_1,
\end{align*}
which, combined with the comparison principle for elliptic equations, yields
\begin{equation}\label{3}
(I-\Delta)^{-1}\{\gamma(v)u+au-bu^\sigma\}\leq-\gamma(0)v+C_1.
\end{equation}
Inserting \eqref{3} into \eqref{6}, we get
$$v_t+\gamma(v)u+\gamma(0)v\leq C_1.$$

Step 3: Owing to the nonnegativity of $u$ and $\gamma(v)$, it holds that
\begin{equation*}
\gamma(v)u\geq0,
\end{equation*}
which, in conjunction with the result obtained in Step 2, leads to
$$v_t+\gamma(0)v\leq C_1.$$
Therefore, by the Gr\"{o}nwall inequality, this completes the proof of the lemma.
\end{proof}

\subsection{The time-space \texorpdfstring{$L^2$}{L2} boundedness of \texorpdfstring{$u$}{u}}

Thanks to the vital inequality of $v$ obtained in Lemma \ref{lm1}, we have the following improved time-space $L^2$ boundedness of $u$ for $\sigma>1$.

\begin{lemma}\label{lm11}
Let $(u,v)$ be a solution of \eqref{1}. If $a>0,\,b>0$ and $\sigma>1$ are constants, then there exists a constant $C>0$ such that
\[\int_\Omega|\nabla v|^2\leq C\quad\text{for all}\ t\in(0,T_{max})\]
and
\begin{equation*}
\int_t^{t+\tau}\int_\Omega u^2\leq C\quad\text{for all}\ t\in(0,T_{max}-\tau),
\end{equation*}
where $\tau=\min\left\{1,\frac{T_{max}}{2}\right\}$.
\end{lemma}

\begin{proof}
Thanks to Lemma \ref{lm1}, we can find a constant $C_1>0$ such that
$$v_t+\gamma(v)u+\gamma(0)v\leq C_1.$$
Then by the second equation of \eqref{1}, integration by parts and Lemma \ref{lm3}, we get
\begin{equation}\label{4}
\begin{aligned}
&\frac{1}{2}\frac{d}{dt}\int_\Omega(|\nabla v|^2+v^2)=\int_\Omega uv_t\\
\leq&-\int_\Omega \gamma(v)u^2-\gamma(0)\int_\Omega uv+C_1\int_\Omega u\\
\leq&-\int_\Omega \gamma(v)u^2-\gamma(0)\int_\Omega(|\nabla v|^2+v^2)+C_2
\end{aligned}
\end{equation}
for some constant $C_2>0$. Therefore, the Gr\"{o}nwall inequality gives a constant $C_3>0$ such that
\[\int_\Omega(|\nabla v|^2+v^2)\leq C_3.\]
Hence, integrating \eqref{4} over $[t,t+\tau]$, we find some constant $C_4>0$ so that
\[\int_t^{t+\tau}\int_\Omega u^2\leq C_4,\]
which gives the desired result.
\end{proof}

\subsection{A criterion  of global existence}

To prove our main result, we shall deduce a criterion  of global boundedness of solutions for the system \eqref{1}. To this end, we first derive an important inequality which will be used in the sequel frequently.

\begin{lemma}\label{lm2}
Let $(u,v)$ be a solution of \eqref{1}. If $a>0,\,b>0$ and $\sigma>1$ are constants, then for any $p\geq\max\{2,\frac{n}{2}-1\}$, one can find constants $C>0$ and $\kappa>0$ which is defined in \eqref{65} such that
$$\begin{aligned}
&\frac{d}{dt}\int_\Omega u^p+\frac{2(p-1)C}{p}\int_\Omega \left|\nabla u^\frac{p}{2}\right|^2+b p\int_\Omega u^{p+\sigma-1}\\
\leq&\frac{p(p-1)\kappa}{2}\int_\Omega u^{p+1}+ap\int_\Omega u^{p}
\quad \text{for all}\ t\in(0,T_{max}).
\end{aligned}$$
\end{lemma}

\begin{proof}
Step 1: We claim that there exist constants $C_1,C_2>0$ such that for any $p\geq2$,
we have
\begin{equation}\label{10}
\begin{aligned}
&\frac{d}{dt}\int_\Omega u^p+\frac{2(p-1)C_1}{p}\int_\Omega \left|\nabla u^\frac{p}{2}\right|^2+b p\int_\Omega u^{p+\sigma-1}\\
\leq&\frac{p(p-1)C_2}{2}\int_\Omega u^p|\nabla v|^2+ap\int_\Omega u^{p}\quad \text{for all}\ t\in(0,T_{max}).
\end{aligned}
\end{equation}
Using $u^{p-1}$ with $p\geq2$ as a test function to test the first equation in \eqref{1}, integrating the result by parts and using Young's inequality, we obtain
\begin{equation*}
\begin{aligned}
&\frac{1}{p}\frac{d}{dt}\int_\Omega u^p+(p-1)\int_\Omega\gamma(v) u^{p-2}|\nabla u|^2+b\int_\Omega u^{p+\sigma-1}\\
=&-(p-1)\int_\Omega\gamma'(v)u^{p-1}\nabla u\cdot\nabla v+a\int_\Omega u^{p}\\
\leq&\frac{p-1}{2}\int_\Omega \gamma(v)u^{p-2}|\nabla u|^2+\frac{p-1}{2}\int_\Omega \frac{|\gamma'(v)|^2}{\gamma(v)}u^p|\nabla v|^2+a\int_\Omega u^{p},
\end{aligned}
\end{equation*}
which, along with the basic fact
$$\left|\nabla u^\frac{p}{2}\right|^2=\frac{p^2}{4}u^{p-2}|\nabla u|^2,$$
yields that
\begin{equation}\label{9}
\begin{aligned}
&\frac{1}{p}\frac{d}{dt}\int_\Omega u^p+\frac{2(p-1)}{p^2}\int_\Omega\gamma(v) \left|\nabla u^\frac{p}{2}\right|^2+b\int_\Omega u^{p+\sigma-1}\\
\leq&\frac{p-1}{2}\int_\Omega \frac{|\gamma'(v)|^2}{\gamma(v)}u^p|\nabla v|^2+a\int_\Omega u^{p}.
\end{aligned}
\end{equation}
In view of hypothesis \eqref{11} and Lemma \ref{lm1}, we can find constants $C_1,\ C_2>0$ such that
$$\gamma(v)\geq\gamma(0)=C_1 \ \ \text{and}\ \ \frac{|\gamma'(v)|^2}{\gamma(v)}\leq C_2.$$
Therefore, \eqref{9} yields \eqref{10}.

Step 2: Now, we estimate the terms on the right hand side of \eqref{10}. First the Young inequality implies
\begin{equation}\label{17}
\int_{\Omega} u^{p}|\nabla v|^{2}\leq \int_{\Omega} u^{p+1}+\int_{\Omega}|\nabla v|^{2(p+1)}.
\end{equation}
In addition since $p\geq\max\{2,\frac{n}{2}-1\}$, we may invoke the Gagliardo-Nirenberg inequality, the regularity theory of elliptic equations and Lemma \ref{lm1} to find some constants $C_{GN},C_R>0$ such that
\begin{align*}
\|\nabla v\|_{L^{2(p+1)}(\Omega)}^{2(p+1)} \leq& G^{2(p+1)}\|v\|_{W^{2, p+1}(\Omega)}^{p+1}\|v\|_{L^{\infty}(\Omega)}^{p+1}\\
\leq& G^{2(p+1)}Q^{p+1}\|v\|_{W^{2, p+1}(\Omega)}^{p+1}\\
\leq& G^{2(p+1)}R^{p+1}Q^{p+1}\|u\|_{L^{p+1}(\Omega)}^{p+1}.
\end{align*}
This along with \eqref{17} updates \eqref{10} as
\begin{align*}
&\frac{d}{dt}\int_\Omega u^p+\frac{2(p-1)C_1}{p}\int_\Omega \left|\nabla u^\frac{p}{2}\right|^2+b p\int_\Omega u^{p+\sigma-1}\\
\leq&\frac{p(p-1)C_2\left( G^{2(p+1)}R^{p+1}Q^{p+1}+1\right)}{2}\int_\Omega u^{p+1}+ap\int_\Omega u^{p}
\end{align*}
which gives the desired result of the lemma.
\end{proof}

The following lemma asserts that, the global solution can be obtained as long as one can find a constant $p_0>\frac{n}{2}$ such that the $L^{p_0}-$norm of $u$ is bounded uniformly in time. The idea of the proof is originally from \cite{Z.A.W.2021MMAS}.

\begin{lemma}\label{lm4}
Let $(u,v)$ be a solution of \eqref{1}. Assume $a>0,\,b>0$ and $\sigma>1$ are constants. If there exist constants $M>0$ and $p_0>\frac{n}{2}$ such that
\begin{equation}\label{20}
\int_{\Omega} u^{p_0} \leq M \ \ \text{for all}\ \ t \in\left(0, T_{max}\right),
\end{equation}
then the solution of the problem \eqref{1} is global, i.e. we can find a constant $C>0$ such that
$$
\|u(\cdot,t)\|_{L^\infty(\Omega)}+\|v(\cdot,t)\|_{W^{1,\infty}(\Omega)}\leq C \ \ \text{for all}\ \ t \in\left(0, T_{max}\right).
$$
\end{lemma}

\begin{proof}
Step 1: We claim that there exists a constant $C>0$ such that
\begin{equation*}
\int_{\Omega} u^{2p_0} \leq C \ \ \text{for all}\ \ t \in\left(0, T_{max}\right).
\end{equation*}
Indeed, taking $p=2p_0$ in Lemma \ref{lm2}, we get the following estimate
\begin{equation}\label{19}
\begin{aligned}
&\frac{d}{dt}\int_\Omega u^{2p_0}+\frac{(2p_0-1)C_1}{p_0}\int_\Omega \left|\nabla u^{p_0}\right|^2+2b p_0\int_\Omega u^{2p_0+\sigma-1}\\
\leq& p_0(2p_0-1)C_2\int_\Omega u^{2p_0+1}+2ap_0\int_\Omega u^{2p_0}
\end{aligned}
\end{equation}
for some constants $C_1,C_2>0$. With Young's inequality, we can find a constant $C_3>0$ such that 
$$
\left(1+2ap_0\right)\int_{\Omega} u^{2p_0}\leq 2b p_0\int_{\Omega} u^{2p_0+\sigma-1}+C_3.
$$
Then we have from \eqref{19} that
\begin{equation}\label{21}
\frac{d}{dt}\int_\Omega u^{2p_0}+\int_\Omega u^{2p_0}+\frac{(2p_0-1)C_1}{p_0}\int_\Omega \left|\nabla u^{p_0}\right|^2\leq p_0(2p_0-1)C_2\int_\Omega u^{2p_0+1}+C_3.
\end{equation}
By the Gagliardo-Nirenberg inequality and \eqref{20}, we can find some constants $C_4,C_5>0$ such that
\begin{equation}\label{22}
\begin{aligned}
\int_{\Omega} u^{2p_0+1}=&\left\|u^{p_0}\right\|_{L^{\frac{2p_0+1}{p_0}}(\Omega)}^{\frac{2p_0+1}{p_0}} \leq C_4\left(\left\|u^{p_0}\right\|_{L^{1}(\Omega)}^{\frac{2p_0+1}{p_0}(1-\theta)}\left\|\nabla u^{p_0}\right\|_{L^{2}(\Omega)}^{\frac{2p_0+1}{p_0} \theta}+\left\|u^{p_0}\right\|_{L^{1}(\Omega)}^{\frac{2p_0+1}{p_0}}\right)\\
\leq&C_4\left(M^{\frac{2p_0+1}{p_0}(1-\theta)}\left\|\nabla u^{p_0}\right\|_{L^{2}(\Omega)}^{\frac{2p_0+1}{p_0} \theta}+M^{\frac{2p_0+1}{p_0}}\right)\\
\leq&\frac{C_1}{p_0^2C_2}\int_\Omega \left|\nabla u^{p_0}\right|^2+C_5,
\end{aligned}
\end{equation}
where $\theta=\frac{n}{n+2} \frac{2p_0+2}{2p_0+1} \in(0,1)$ and $\frac{2p_0+1}{2p_0}\theta<1$ due to $p_0>\frac{n}{2}$. Substituting \eqref{22} into \eqref{21}, we obtain
$$\frac{d}{dt}\int_\Omega u^{2p_0}+\int_\Omega u^{2p_0}\leq p_0(2p_0-1)C_2 C_5+C_3$$
which, together with the Gr\"{o}nwall's inequality, proves the claim.

Step 2: Now noticing the estimate in Step 1 and the regularity theory of elliptic equations, we get from the second equation of \eqref{1} that $v \in W^{2,2p_0}(\Omega) \hookrightarrow C^{1,1-\frac{n}{2p_0}}(\Omega)$ by the Sobolev embedding theorem. In other words, there exists a constant $C_6>0$ such that
\begin{equation}\label{40}
\|\nabla v\|_{L^{\infty}(\Omega)} \leq C_6\quad \text{for all}\ t\in\left(0, T_{max}\right)
\end{equation}
which, together with Lemma \ref{lm2}, implies that for any $p\geq2$
$$
\frac{d}{d t} \int_{\Omega} u^{p}+\frac{(p-1)C_7}{p} \int_{\Omega}\left|\nabla u^{\frac{p}{2}}\right|^{2}+bp\int_\Omega u^{p+\sigma-1}\leq C_8p(p-1)\int_{\Omega} u^{p}+ap \int_{\Omega} u^{p}.
$$
with some constants $C_7,C_8>0$. Then, we can employ the  standard Moser iteration (cf. \cite{N.D.A.1979CPDE}) or some similar arguments as in \cite{T-W2013MMMAS} to prove that there exists a constant $C_9>0$ such that
\begin{equation}\label{41}
\|u\|_{L^{\infty}(\Omega)} \leq C_9\quad \text{for all}\ t \in(0, T_{max})
\end{equation}
where the details are omitted here for brevity. Then we complete the proof by using \eqref{40}, \eqref{41} and the extensibility criterion in Lemma \ref{lm5}.
\end{proof}

\section{Global existence: Proof of Theorem \ref{th4}}\label{sec3}

Now, we shall use the boundedness criterion  in Lemma \ref{lm4} to show \eqref{1} admits a global classical solution. If $\sigma>2$, then we use the Young inequality to get the following estimate.

\begin{lemma}\label{lm7}
Let $(u,v)$ be a solution of \eqref{1}. If $a>0,\,b>0$, $\sigma>2$ and $\gamma(v)$ satisfies the hypothesis \eqref{11}, then for any $p\geq\max\{2,\frac{n}{2}-1\}$ there exists a constant $C>0$ such that
$$
\int_{\Omega} u^p \leq C \ \ \text{for all}\ \ t \in\left(0, T_{max}\right).
$$
\end{lemma}

\begin{proof}
By Lemma \ref{lm2}, we get the following estimate
\begin{equation}\label{18}
\frac{d}{dt}\int_\Omega u^p+\frac{2(p-1)C_1}{p}\int_\Omega \left|\nabla u^\frac{p}{2}\right|^2+b p\int_\Omega u^{p+\sigma-1}\leq\frac{p(p-1)C_2}{2}\int_\Omega u^{p+1}+ap\int_\Omega u^{p}
\end{equation}
for some constants $C_1,C_2>0$. An application of the Young inequality gives some constant $C_3>0$ such that
$$\frac{p(p-1)C_2}{2}\int_\Omega u^{p+1}+\left(1+ap\right)\int_\Omega u^{p}\leq b p\int_\Omega u^{p+\sigma-1}+C_3,$$
which, combined with \eqref{18}, yields
$$\frac{d}{dt}\int_\Omega u^p+\int_\Omega u^p\leq C_3.$$
Hence, we can use the Gr\"{o}nwall inequality to prove the claim.
\end{proof}

Similarly, if $\sigma=2$ and $b>0$ is large enough, then the uniform in-time $L^{\left[\frac n2\right]+1}-$norm of $u$ can be obtained.

\begin{lemma}\label{lm16}
Let $(u,v)$ be a solution of \eqref{1}. If $a>0$, $b>b_1$, $\sigma=2$ and $\gamma(v)$ satisfies the hypothesis \eqref{11}, then we can find a constant $C>0$ such that
$$
\int_{\Omega} u^{\left[\frac n2\right]+1} \leq C \ \ \text{for all}\ \ t \in\left(0, T_{max}\right).
$$
\end{lemma}

\begin{proof}
By Lemma \ref{lm2}, for some $\widetilde{p}=\left[\frac{n}{2}\right]+1$, there exists some constants $C_1>0$ and $\kappa(\widetilde{p},b,2)>0$ defined in \eqref{65} such that
\begin{equation}\label{60}
\frac{d}{dt}\int_\Omega u^{\widetilde{p}}+\frac{2(\widetilde{p}-1)C_1}{\widetilde{p}}\int_\Omega \left|\nabla u^\frac{\widetilde{p}}{2}\right|^2+b \widetilde{p}\int_\Omega u^{\widetilde{p}+1}\leq\frac{\widetilde{p}(\widetilde{p}-1)\kappa(\widetilde{p},b,2)}{2}\int_\Omega u^{\widetilde{p}+1}+a\widetilde{p}\int_\Omega u^{\widetilde{p}}.
\end{equation}
The Young inequality gives a constant $C_2>0$ satisfying
\begin{equation}\label{61}
\left(1+a\widetilde{p}\right)\int_\Omega u^{\widetilde{p}}\leq \widetilde{p}\int_\Omega u^{\widetilde{p}+1}+C_2.
\end{equation}
Taking $b\geq b_1,$ we have from the definition of $\kappa$ that
$$b\geq\frac{(\widetilde{p}-1)\kappa(\widetilde{p},b,2)}{2}+1,$$
which, along with \eqref{60} and \eqref{61}, implies
$$\frac{d}{dt}\int_\Omega u^{\widetilde{p}}+\int_\Omega u^{\widetilde{p}}\leq C_2.$$
Hence, we get the desired result by the Gr\"{o}nwall inequality.
\end{proof}

If $n=2$, we make full use of the second equation of \eqref{1} to get the uniform $L^2$-norm of $u$.

\begin{lemma}\label{lm6}
Let $(u,v)$ be a solution of \eqref{1}. If $n=2$, $a>0,\,b>0$ and $\sigma>1$ are constants, then there exists a constant $C>0$ such that
$$
\int_{\Omega} u^2 \leq C \ \ \text{for all}\ \ t \in\left(0, T_{max}\right).
$$
\end{lemma}

\begin{proof}
We multiply the second equation of \eqref{1} by $-\Delta v$ and integrate the result by parts to have
\begin{equation*}
\int_\Omega |\Delta v|^2+\int_\Omega|\nabla v|^2=-\int_\Omega u\Delta v\leq\frac{1}{2}\int_\Omega |\Delta v|^2+\frac{1}{2}\int_\Omega u^2,
\end{equation*}
which yields
\begin{equation}\label{35}
\frac{1}{2}\int_\Omega |\Delta v|^2+\int_\Omega|\nabla v|^2\leq\frac{1}{2}\int_\Omega u^2.
\end{equation}
Taking $p=2$ in \eqref{10}, we get the following estimate
\begin{equation}\label{36}
\frac{d}{dt}\int_\Omega u^2+C_1\int_\Omega \left|\nabla u\right|^2+2b\int_\Omega u^{\sigma+1}\leq C_2\int_\Omega u^2|\nabla v|^2+2a\int_\Omega u^2
\end{equation}
for some constants $C_1,C_2>0$. Using \eqref{35}, Lemma \ref{lm1}, Lemma \ref{lm11}, H\"{o}lder's inequality, the Gagliardo-Nirenberg inequality (see \cite[Lemma 2.5]{J-K-W2018SIAM}) and Young's inequality, we can find constants $C_3,C_4,C_5,C_6>0$ such that
\begin{equation*}
\begin{aligned}
C_2\int_\Omega u^2|\nabla v|^2\leq&C_2\|u\|_{L^4(\Omega)}^2\|\nabla v\|_{L^4(\Omega)}^2\\
\leq&C_3\left(\|\nabla u\|_{L^2(\Omega)}^{\frac{n}{4}}\|u\|_{L^2(\Omega)}^{\frac{4-n}{4}}+\|u\|_{L^1(\Omega)}\right)^2\left(\|\Delta v\|_{L^2(\Omega)}^{\frac{1}{2}}\|v\|_{L^\infty(\Omega)}^{\frac{1}{2}}+\|v\|_{L^\infty(\Omega)}\right)^2\\
\leq&C_4\left(\|\nabla u\|_{L^2(\Omega)}^{\frac{n}{2}}\|u\|_{L^2(\Omega)}^{\frac{4-n}{2}}\|\Delta v\|_{L^2(\Omega)}+\|\nabla u\|_{L^2(\Omega)}^{\frac{n}{2}}\|u\|_{L^2(\Omega)}^{\frac{4-n}{2}}+\|\Delta v\|_{L^2(\Omega)}+1\right)\\
\leq&C_1\|\nabla u\|_{L^2(\Omega)}^2+C_5\left(1+\|\Delta v\|_{L^2(\Omega)}^{\frac{4}{4-n}}\right)\|u\|_{L^2(\Omega)}^2+C_5\\
\leq&C_1\|\nabla u\|_{L^2(\Omega)}^2+C_6\left(1+\|u\|_{L^2(\Omega)}^2\right)\|u\|_{L^2(\Omega)}^2+C_6,
\end{aligned}
\end{equation*}
which, combined with \eqref{36}, implies
$$\frac{d}{dt}\int_\Omega u^2\leq C_6\left(1+\|u\|_{L^2(\Omega)}^2\right)\int_\Omega u^2+C_7$$
for some constant $C_7>0$. An application of Lemma \ref{lm9} and Lemma \ref{lm11} gives the desired result.
\end{proof}

\begin{proof}[Proof of Theorem \ref{th4}]
Consolidating the results Lemma \ref{lm3}, Lemma \ref{lm4}, Lemma \ref{lm7},  Lemma \ref{lm16} and Lemma \ref{lm6}, we get Theorem \ref{th4} directly.
\end{proof}

\section{Large time behavior}\label{sec4}
In this section, we aim to study the large time behavior of \eqref{1}. We first improve the regularity of the solution component $u$.

\begin{lemma}\label{lm13}
There exist constants $\theta\in(0,1)$ and $C>0$ such that
$$\|u\|_{C^{\theta,\frac{\theta}{2}}(\overline{\Omega}\times[t,t+1])}\leqslant C\quad\text{for all}\ t>1.$$
\end{lemma}

\begin{proof}
Let $\psi_0=C\gamma'^2(v)|\nabla v|^2$, $\psi_1=\gamma'(v)\nabla v$ and $\psi_2=a u+bu^\sigma$ in conditions $(A_1),\,(A_2),\,(A_3)$ of \cite{P-V1993}. With an application of the results in \cite{P-V1993} to the solution of the first equation of \eqref{1} and the boundedness of $u,\ v$ and $\nabla v$, we get the results directly.
\end{proof}

Based on the Lyapunov functionals method, the large time behavior of the solution can be studied. We first show a basic property on the solution component $u$.

\begin{lemma}\label{lm8}
For any $\alpha>1$, there exists a constant $T_*>0$ such that
\begin{equation}\label{42}
\int_\Omega u\leq\alpha u_*|\Omega|\quad\text{for all}\ t>T_*,
\end{equation}
where $u_*$ is defined in \eqref{33}.
\end{lemma}

\begin{proof}
Integrating the first equation in \eqref{1} over $\Omega$ and using the boundary condition as well as H\"{o}lder's inequality, we have
\begin{equation*}
\frac{d}{dt}\int_\Omega u\leq a\int_\Omega u-\frac{b}{|\Omega|^{\sigma-1}}\left(\int_\Omega u\right)^\sigma\quad \text{for all}\ t\geqslant0.
\end{equation*}
In order to solve this ordinary differential inequality, we first consider the following ordinary differential equation
\begin{equation*}
\begin{cases}
\frac{d}{dt}y(t)=ay(t)-\frac{b}{|\Omega|^{\sigma-1}}y^\sigma(t),\quad t\geqslant0,\\
y(0)=\int_\Omega u_0.
\end{cases}
\end{equation*}
which is a Bernoulli's equation having the solution
$$y(t)=\left\{\frac{b}{a|\Omega|^{\sigma-1}}+\left[\left(\int_\Omega u_0\right)^{-(\sigma-1)}-\frac{b}{a|\Omega|^{\sigma-1}}\right] e^{-a(\sigma-1) t}\right\}^{-\frac{1}{\sigma-1}}.$$
This, along with a comparison argument, shows that
\begin{equation}\label{43}
\int_\Omega u\leq\left\{\frac{b}{a|\Omega|^{\sigma-1}}+\left[\left(\int_\Omega u_0\right)^{-(\sigma-1)}-\frac{b}{a|\Omega|^{\sigma-1}}\right] e^{-a(\sigma-1) t}\right\}^{-\frac{1}{\sigma-1}}.
\end{equation}
Observing
$$\left\{\frac{b}{a|\Omega|^{\sigma-1}}+\left[\left(\int_\Omega u_0\right)^{-(\sigma-1)}-\frac{b}{a|\Omega|^{\sigma-1}}\right] e^{-a(\sigma-1) t}\right\}^{-\frac{1}{\sigma-1}}\rightarrow u_*|\Omega|$$
as $t\rightarrow+\infty$, we can find a constant $T_*>0$ such that
$$\left\{\frac{b}{a|\Omega|^{\sigma-1}}+\left[\left(\int_\Omega u_0\right)^{-(\sigma-1)}-\frac{b}{a|\Omega|^{\sigma-1}}\right] e^{-a(\sigma-1) t}\right\}^{-\frac{1}{\sigma-1}}\leq\alpha u_*|\Omega|\quad\text{for all}\ t>T_*.$$
Together with \eqref{43}, we get  \eqref{42}.
\end{proof}

We note that the result in Lemma \ref{lm8} slightly improves the result of  \cite[Lemma 2.2]{M.W.2020ANS} where $\alpha=2$. To prove the large time behavior of solutions, we shall prove an inequality below, which looks similar to \cite[Lemma 3.4]{M.W.2020ANS} derived for the classical logistic  Keller-Segel system  without density-suppressed motility. Here we extend this inequality to the density-suppressed motility model with the logistic source. Since the proof has some differences from \cite{M.W.2020ANS},  we provide the details for clarity.

\begin{lemma}\label{lm12}
Let
$$q:=\begin{cases}
\frac{2}{3-\sigma},&\text{if}\ \ \sigma\in(1,2),\\
2,&\text{if}\ \ \sigma\in[2,+\infty).
\end{cases}$$
Then there exist constants $K_1>0$ which is defined in \eqref{66}and $T_*\geq0$ such that
\begin{equation}\label{1-15}
\left\|u-u_*\right\|_{L^{q}(\Omega)}^{2} \leq K_{1}u_*^{2-\sigma} \int_{\Omega}\left(u^{\sigma-1}-u_*^{\sigma-1}\right) \left(u-u_*\right) \quad \text{for all}\ \ t>T_*
\end{equation}
where $u_*>0$ is defined in \eqref{33}.
\end{lemma}

\begin{proof}
Step 1: If $\sigma\in[2,+\infty)$, then $u^{\sigma-2} \leq u_*^{\sigma-2}$ for $u \in[0,u_*]$, and thus
\begin{equation}\label{1-5}
\begin{aligned}
\left(u-u_*\right)\left(u^{\sigma-1}-u_*^{\sigma-1}\right)=&\left(u_*-u\right)\left(u_*^{\sigma-1}-u^{\sigma-1}\right)\\
\geq&\left(u_*-u\right)\left(u_*^{\sigma-1}-u_*^{\sigma-2}\right)\\
=&u_*^{\sigma-2}\left(u-u_*\right)^2\quad \text{for all}\ u \in[0,u_*]
\end{aligned}
\end{equation}
whereas $u_*^{\sigma-2} \leq u^{\sigma-2}$ for $u \in(u_*,+\infty)$, and hence
\begin{align*}
\left(u-u_*\right)\left(u^{\sigma-1}-u_*^{\sigma-1}\right)\geq&\left(u -u_*\right)\left(u_*^{\sigma-2}-u_*^{\sigma-1}\right)
=u_*^{\sigma-2}\left(u-u_*\right)^2\quad \text{for all}\ u\in(u_*,+\infty).
\end{align*}
This, along with \eqref{1-5}, establishes
$$\left(u-u_*\right)\left(u^{\sigma-1}-u_*^{\sigma-1}\right)\geq u_*^{\sigma-2}\left(u-u_*\right)^2\quad \text{for any}\ \ \sigma\in[2,+\infty),$$
which implies \eqref{1-15}.

Step 2: If $\sigma\in(1,2)$, then $\varphi(\xi):=\xi^{\sigma-1}$ is a continuous function on $[0,2u_*]$ and $\varphi'(\xi)=(\sigma-1)\xi^{\sigma-2}$ on $(0,2u_*)$. For $u\in[0,u_*]$, the mean value theorem yields that there exists $\theta\in(u,u_*)$ such that
\begin{align*}
u^{\sigma-1}-u_*^{\sigma-1}=&(\sigma-1)\theta^{\sigma-2}\left(u-u_*\right)
\leq (\sigma-1)u_*^{\sigma-2}\left(u-u_*\right),
\end{align*}
which implies
\begin{equation}\label{1-7}
(u-u_*)(u^{\sigma-1}-u_*^{\sigma-1})\geq (\sigma-1)u_*^{\sigma-2}(u-u_*)^2.
\end{equation}
Meanwhile, for $u\in(u_*,2u_*]$, an application of the mean value theorem gives $\vartheta\in(u_*,u)$ satisfying
\begin{align*}
u^{\sigma-1}-u_*^{\sigma-1}=&(\sigma-1)\vartheta^{\sigma-2}\left(u-u_*\right)
\geq (\sigma-1)(2u_*)^{\sigma-2}\left(u-u_*\right),
\end{align*}
which yields
$$(u-u_*)(u^{\sigma-1}-u_*^{\sigma-1})\geq (\sigma-1)(2u_*)^{\sigma-2}(u-u_*)^2.$$
This together with \eqref{1-7} proves
\begin{equation}\label{47}
(u-u_*)(u^{\sigma-1}-u_*^{\sigma-1})\geq (\sigma-1)(2u_*)^{\sigma-2}(u-u_*)^2\quad\text{for any}\ \ u\in[0,2u_*]
\end{equation}
due to $2^{\sigma-2}\leq1$.

Since $q \leq 2$, we have $(s+r)^{\frac{2}{q}} \leq 2^{\frac{2}{q}-1}\left(s^{\frac{2}{q}}+r^{\frac{2}{q}}\right)$ for $s \geq 0$ and $r \geq 0$ which yields
\begin{equation}\label{1-12}
\begin{aligned}
\left\|u-u_*\right\|_{L^{q}(\Omega)}^{2}=&\left(\int_{\left\{u<2 u_*\right\}}\left|u-u_*\right|^{q}+\int_{\left\{u\geq 2 u_*\right\}}\left(u-u_*\right)^{q}\right)^{\frac{2}{q}}\\
\leq& 2^{\frac{2}{q}-1}\left(\int_{\left\{u<2 u_*\right\}}\left|u-u_*\right|^{q}\right)^{\frac{2}{q}}+2^{\frac{2}{q}-1} \left(\int_{\left\{u\geq 2 u_*\right\}} u^{q}\right)^{\frac{2}{q}}
\end{aligned}
\end{equation}
for all $t\geq0$. Now we estimate the inequalities on the right hand side of \eqref{1-12}.

For the first integral, noticing again $q \leq 2$, we invoke the H\"{o}lder inequality and \eqref{47} to find that for all $t\geq 0$,
\begin{equation}\label{1-13}
\begin{aligned}
2^{\frac{2}{q}-1}\left(\int_{\left\{u<2 u_*\right\}}\left|u-u_*\right|^{q}\right)^{\frac{2}{q}} \leq& 2^{\frac{2}{q}-1}|\Omega|^{\frac{2-q}{q}} \int_{\left\{u<2 u_*\right\}}\left(u-u_*\right)^{2}\\
\leq&\frac{(2|\Omega|)^{\frac{\sigma+1}{3-\sigma}}(2u_*)^{2-\sigma}}{\sigma-1}\int_{\left\{u<2 u_*\right\}}\left(u^{\sigma-1}-u_*^{\sigma-1}\right) \left(u-u_*\right).
\end{aligned}
\end{equation}
Thanks to Lemma \ref{lm8}, we can pick a constant $T_*>0$ satisfying
\begin{equation}\label{1-9}
\left\|u\right\|_{L^1(\Omega)}\leq2u_*|\Omega|\quad\text{for all}\ t>T_*.
\end{equation}
In order to estimate the second integral, we may utilize \eqref{1-9} to see that if $\sigma<2,$ it follows from the H\"{o}lder inequality and our special choice of $q$ that
\begin{equation}\label{1-10}
\begin{aligned}
\left(\int_{\left\{u\geq 2 u_*\right\}} u^{q}\right)^{\frac{2}{q}}=\left\|u\right\|_{L^{\frac{2}{3-\sigma}}\left(\left\{u \geq2 u_*\right\}\right)}^{2}
\leq&\left\|u\right\|_{L^{1}\left(\left\{u\geq 2 u_*\right\}\right)}^{2-\sigma}\left\|u\right\|_{L^{\sigma}\left(\left\{u\geq 2 u_*\right\}\right)}^{\sigma} \\
\leq&\left(2 u_*|\Omega|\right)^{2-\sigma} \int_{\left\{u\geq 2 u_*\right\}} u^\sigma\quad\text{for all}\ t>T_*.
\end{aligned}
\end{equation}
Simple calculation shows
\begin{align*}
\int_{\left\{u\geq2 u_*\right\}}\left(u^{\sigma-1}-u_*^{\sigma-1}\right)\left(u-u_*\right) \geq& \int_{\left\{u\geq2 u_*\right\}}\left(u^{\sigma-1}-\left(\frac{u}{2}\right)^{\sigma-1}\right) \left(u-\frac{u}{2}\right)\\
=&\frac{1-2^{-(\sigma-1)}}{2} \int_{\left\{u\geq 2 u_*\right\}} u^\sigma\quad\text{for all}\ t\geq0,
\end{align*}
which, along with \eqref{1-10}, indicates that
\begin{equation}\label{1-14}
2^{\frac{2}{q}-1} \left(\int_{\left\{u\geq2 u_*\right\}} u^{q}\right)^{\frac{2}{q}}\leq \frac{2^{\frac{4}{3-\sigma}}(2u_*|\Omega|)^{2-\sigma}}{1-2^{-(\sigma-1)}}\int_{\left\{u \geq2 u_*\right\}}\left(u^{\sigma-1}-u_*^{\sigma-1}\right)\left(u-u_*\right)
\end{equation}
for all $t>T_*$.

Collecting \eqref{1-12}, \eqref{1-13} and \eqref{1-14}, we obtain \eqref{1-15}.
\end{proof}

Next we derive an estimate on $v$.

\begin{lemma}\label{lm10}
We have
$$\int_\Omega|\nabla v|^2+\int_\Omega(v-\overline{v})^2\leq K_2\|u-u_*\|_{L^q(\Omega)}^2,$$
where $q$ is defined in Lemma \ref{lm12} and $K_2>0$ is defined in \eqref{64}.
\end{lemma}

\begin{proof}
Due to
$$\overline{v}=\frac{1}{|\Omega|}\int_\Omega v,$$
it holds that
\begin{equation}\label{48}
\int_\Omega(v-\overline{v})=0.
\end{equation}
Multiplying the second equation of \eqref{1} by $v-\overline{v}$ and integrating the result over $\Omega$, we get
\begin{equation}\label{44}
\int_\Omega|\nabla(v-\overline{v})|^2+\int_\Omega(v-\overline{v})v=\int_\Omega(v-\overline{v})u,
\end{equation}
which, along with \eqref{48}, updates \eqref{44} as
\begin{equation}\label{49}
\int_\Omega|\nabla v|^2+\int_\Omega(v-\overline{v})^2=\int_\Omega(v-\overline{v})(u-u_*).
\end{equation}
We proceed with two cases. If $\sigma\in[2,+\infty)$, the Young inequality implies
$$\int_\Omega(v-\overline{v})(u-u_*)\leq\frac{1}{2}\int_\Omega(v-\overline{v})^2+\frac{1}{2}\|u-u_*\|_{L^q(\Omega)}^2$$
which, along with \eqref{49}, yields the desired result. 
If $1<\sigma<2$ with 
$$\sigma\in\begin{cases}\left(1,2\right),&\text{if}\ n\leq2,\\
\left(2-\frac2n,2\right),&\text{if}\ n>2,
\end{cases}$$
thanks to the definition of $q$, we obtain from Poincar\'{e}'s inequality and the Sobolev imbedding theorem that there exists a constant $\xi>0$ such that
\begin{equation*}
\|v-\overline{v}\|_{L^{\frac{q}{q-1}}(\Omega)}\leq\xi\|\nabla v\|_{L^2(\Omega)},
\end{equation*}
which, along with H\"{o}lder's inequality and the Young inequality, yields
\begin{equation*}
\begin{aligned}
\int_\Omega(v-\overline{v})(u-u_*)\leq&\|v-\overline{v}\|_{L^{\frac{q}{q-1}}(\Omega)}\|u-u_*\|_{L^q(\Omega)}\\
\leq&C_p\|\nabla v\|_{L^2(\Omega)}\|u-u_*\|_{L^q(\Omega)}\\
\leq&\frac{1}{2}\|\nabla v\|_{L^2(\Omega)}^2+\frac{\xi^2}{2}\|u-u_*\|_{L^q(\Omega)}^2.
\end{aligned}
\end{equation*}
Therefore, we use \eqref{49} to finish the proof of the lemma.
\end{proof}

To study the large time behavior of solutions, we shall construct a Lyapunov functional for \eqref{1} inspired by works \cite{J-K-W2018SIAM, T-W2015SIAM}. Given a positive number $u_*$, let $\varphi_{u_*}: (0,\infty)\rightarrow\mathbb{R}$ be defined by
$$\varphi_{u_*}(x):=x-u_*-u_*\ln\frac{x}{u_*},\quad x>0.$$
Then $\varphi_{u_*}$ is convex with $\varphi_{u_*}(u_*)=\varphi'_{u_*}(u_*)=0$, which implies $\varphi_{u_*}(x)\geq0$ for all $x>0$. For any nonnegative continuous function $u:\overline{\Omega}\rightarrow(0,\infty)$, we define an energy functional for \eqref{1} as follows:
\begin{equation*}
{F}(u)=\int_\Omega\left(u-u_*-u_*\ln\frac{u}{u_*}\right)
\end{equation*}
which is nonnegative and non-increasing. We further have the following results.

\begin{lemma}\label{lm15}
If $b>b_2$ where $b_2$ is defined in \eqref{62}, then there exists a constant $C>0$ such that
$$\frac{d}{dt}F(u)+C \left\{\int_\Omega(v-\overline{v})^2+\|u-u_*\|_{L^q(\Omega)}^2\right\}\leq 0\quad\text{for all}\ t>T^*,$$
where $u_*$ is defined as in \eqref{33}. Moreover, there exists a constant $C>0$ such that
$$\int_{T^*}^{+\infty}\int_\Omega\left(v-\overline{v}\right)^2+\int_{T^*}^{+\infty}\left\|u-u_*\right\|^2_{L^q(\Omega)}\leq C.$$
\end{lemma}

\begin{proof}
Noticing Lemma \ref{lm12} and Lemma \ref{lm10}, one can find constants $K_1,K_2>0$ such that
\begin{equation}\label{46}
\left\|u-u_*\right\|_{L^{q}(\Omega)}^{2} \leq K_{1}u_*^{2-\sigma} \int_{\Omega}\left(u^{\sigma-1}-u_*^{\sigma-1}\right) \left(u-u_*\right) \quad \text{for all}\ \ t>T_*
\end{equation}
and
\begin{equation}\label{30}
\int_\Omega|\nabla v|^2+\int_\Omega(v-\overline{v})^2\leq K_2\|u-u_*\|_{L^q(\Omega)}^2.
\end{equation}
An application of the first equation of the system \eqref{1} and integration by parts, we find
\begin{equation}\label{32}
\begin{aligned}
&\frac{d}{dt}\int_\Omega\left(u-u_*-u_*\ln\frac{u}{u_*}\right)
=\int_\Omega u_t-u_*\int_\Omega \frac{u_t}{u}\\
=&a\int_\Omega u-b\int_\Omega u^\sigma-u_*\int_\Omega\frac{1}{u}\left(\Delta(\gamma (v)u)+au-b u^\sigma\right)\\
=&a\int_\Omega u-b\int_\Omega u^\sigma-au_*|\Omega|+bu_*\int_\Omega u^{\sigma-1}\\
&\quad-u_*\int_\Omega\gamma (v)\frac{|\nabla u|^2}{u^2}-u_*\int_\Omega\gamma'(v)\frac{\nabla u\cdot\nabla v}{u}\\
=&-b\int_\Omega(u^{\sigma-1}-u_*^{\sigma-1})(u-u_*)\\
&\quad-u_*\int_\Omega\gamma (v)\frac{|\nabla u|^2}{u^2}-u_*\int_\Omega\gamma'(v)\frac{\nabla u\cdot\nabla v}{u}.
\end{aligned}
\end{equation}
Due to $b>b_2$, we have
$$\frac{|\gamma'(v)|^2u_*}{4\gamma(v)}<\frac{bu_*^{\sigma-2}}{K_1K_2},$$
which implies that there exists a constant $\varepsilon>0$ such that
\begin{equation}\label{51}
\frac{|\gamma'(v)|^2u_*}{4\gamma(v)}\leq\varepsilon<\frac{bu_*^{\sigma-2}}{K_1K_2}.
\end{equation}
Combining \eqref{46}, \eqref{30} and \eqref{32}, one gets
\begin{equation*}
\begin{aligned}
&\frac{d}{dt}\int_\Omega\left(u-u_*-u_*\ln\frac{u}{u_*}\right)\\
\leq&-b\int_\Omega(u^{\sigma-1}-u_*^{\sigma-1})(u-u_*)-\varepsilon\int_\Omega(v-\overline{v})^2+\varepsilon K_2\|u-u_*\|_{L^q(\Omega)}^2\\
&\quad-u_*\int_\Omega\gamma (v)\frac{|\nabla u|^2}{u^2}-u_*\int_\Omega\gamma'(v)\frac{\nabla u\cdot\nabla v}{u}-\varepsilon\int_\Omega|\nabla v|^2\\
\leq&-\left(\frac{b u_*^{\sigma-2}}{K_{1}}-\varepsilon K_2\right)\left\|u-u_*\right\|_{L^{q}(\Omega)}^{2}-\varepsilon\int_\Omega(v-\overline{v})^2-\int_\Omega \Sigma^TB\Sigma,
\end{aligned}
\end{equation*}
where
$$\Sigma=\begin{bmatrix}\frac{\nabla u}{u} \\ \nabla v \end{bmatrix},\quad B=\begin{bmatrix} \gamma(v)u_* & \frac{1}{2}\gamma'(v)u_* \\ \frac{1}{2}\gamma'(v)u_* & \varepsilon \end{bmatrix}.$$
As \eqref{51} ensures that $\frac{b u_*^{\sigma-2}}{K_{1}}-\varepsilon K_2$ is  positive,  $B$ and $F(u)$ is non-negative, the results of Lemma \ref{lm15} follow directly.
\end{proof}

Now, we use the argument of contradiction, Lemma \ref{lm13} and Lemma \ref{lm15} to obtain the asymptotic behavior of \eqref{1} and prove Theorem \ref{th3}.

\begin{proof}[Proof of Theorem \ref{th3}]
We use the argument of contradiction to prove the theorem. Supposing that $\|u(\cdot,t)-u_*\|_{L^\infty}\not \rightarrow0$, we can find a constant $C_1>0$, some sequences
$\{t_j\}_{j\in\mathbb{N}}\subset(1,+\infty)$ such that $t_j\rightarrow+\infty$ as $j\rightarrow+\infty$ and
$$\left\|u(\cdot,t_j)-u_*\right\|_{L^\infty}\geq C_1$$
for all $j\in\mathbb{N}$. Without loss of generality, we may assume that $t_{j+1}>t_j+1$ for all $j\in\mathbb{N}$. Since $\{u(\cdot,t_j)-u_*\}_{j\in\mathbb{N}}$ is relatively compact in $C^0(\overline{\Omega})$ according to Lemma \ref{lm13} and the Arzel\`{a}-Ascoli theorem, we may assume that
$$u(\cdot,t_j)-u_*\rightarrow \widetilde{u}\quad\text{in}\quad L^\infty(\Omega)$$
as $j\rightarrow+\infty$ with some nonnegative $\widetilde{u}\in C^0(\overline{\Omega})$, where the subsequence of $\{u(\cdot,t_j)-u_*\}_{j\in\mathbb{N}}$ is denoted by $\{u(\cdot,t_j)-u_*\}_{j\in\mathbb{N}}$ for notational convenience. Hence we have
$$u(\cdot,t_j)-u_*\rightarrow \widetilde{u}\quad\text{in}\quad L^q(\Omega)$$
as $j\rightarrow+\infty$ which implies
$$\left\|u(\cdot,t_j)-u_*\right\|_{L^q(\Omega)}\geq\frac{1}{2}\| \widetilde{u}\|_{L^q(\Omega)}$$
for all sufficient large $j\in\mathbb{N}$. Now using Lemma \ref{lm13}, we easily get the existence of $\tau\in(0,1)$ and $j_0\in\mathbb{N}$ such that
$$\left\|u(\cdot,t)-u_*\right\|_{L^q(\Omega)}\geq\frac{1}{4}\|\widetilde{u}\|_{L^q(\Omega)}$$
for all $t\in[t_j,t_{j}+\tau]$ and each $j\geqslant j_0$ which in particular implies that for any $j\geqslant j_0$ we have
\begin{equation}\label{6.11}
\int_{t_j}^{t_j+\tau}\left\|u(x,t)-u_*\right\|^2_{L^q(\Omega)}\geq\frac{\tau}{16}\|\widetilde{u}\|^2_{L^q(\Omega)}.
\end{equation}
From Lemma \ref{lm15}, we derive
\begin{align*}
\int_{t_j}^{t_{j}+\tau}\left\|u(x,t)-u_*\right\|^2_{L^q(\Omega)}\leq\int_{t_j}^{+\infty}\left\|u(x,t)-u_*\right\|^2_{L^q(\Omega)}\rightarrow0
\end{align*}
as $j\rightarrow+\infty$ which contradicts \eqref{6.11}. Hence, \eqref{1.5} is verified.

The convergence $\|v(\cdot,t)-u_*\|_{L^\infty}\rightarrow0$ as $t \to \infty$ can be directly obtained by the maximum principle applied to the second equation of \eqref{1}.
\end{proof}

\addcontentsline{toc}{section}{Acknowledgments}
\section*{Acknowledgments}
The research of the first author was supported by the National Nature Science Foundation of China (No. 12101377). The research of the second author was supported by the Hong Kong RGC GRF grant No. 15303019 (Project ID P0030816) and an internal grant No. UAH0 (Project ID P0031504) from the Hong Kong Polytechnic University.

\addcontentsline{toc}{section}{References}

\end{document}